\documentclass{article}

\usepackage{amsmath}
\usepackage{amsfonts}
\usepackage{stmaryrd}
\usepackage{graphicx}
\usepackage{amsthm}
\usepackage[colorlinks=true, allcolors=blue]{hyperref}

\title{On \KURKA{}'s dichotomy for cellular automata on groups}
\author{Jade Angela Hope Audouard\\ \textit{I2M, CNRS, Aix-Marseille University} \and Guillaume Theyssier\\ \textit{I2M, CNRS, Aix-Marseille University}}

\begin{document}

\newcommand\p{ p }
\newcommand\porter{$\Delta$}
\newcommand\Z{\mathbb{Z}}
\newcommand\N{\mathbb{N}}
\newcommand\KURKA{K\r{u}rka}

\maketitle

\begin{abstract}
  This paper is about topological dynamics of cellular automata on finitely generated groups.
  We tackle the problem of determining for which group sensitivity to initial conditions is equivalent to the absence of equicontinuity points (so-called \KURKA{}'s dichotomy).
  We show that the dichotomy holds on any virtually-$\Z$ group but not on free groups with 2 or more generators.
  We also show that if it holds on some group, it must hold on any of its subgroups.
\end{abstract}

\theoremstyle{plain}
\newtheorem{theorem}{Theorem}
\newtheorem{corollary}{Corollary}
\newtheorem{proposition}{Proposition}
\newtheorem{lemma}{Lemma}

\theoremstyle{definition}
\newtheorem{figuree}{Figure}
\newtheorem{definition}{Definition}
\newtheorem{bliblop}{ceciestuntest}

\theoremstyle{remark}
\newtheorem*{demonstration}{Démonstration}
\newtheorem{remarque}{Remarque}[subsection]
\newtheorem{legende}{Légende}[figuree]

\section{Introduction}
Cellular automata have been studied as topological dynamical systems since the seminal work of Hedlund {\it et al.\/} \cite{hedlund}.
In particular, they represent a class of systems where chaotic behaviors are common and very simple to define, which justifies the study of classical properties like sensitivity to initial conditions and presence of equicontinuity points.
Recall that sensitivity to initial conditions is a property of instability of the system (every open set grows to at least a fixed diameter), while equicontinuity points are those whose orbit is stable (perturbed orbits stay arbitrarily close to the point's orbit provided small enough perturbation on initial configuration). These two properties are mutually exclusive and in fact there is a dichotomy between the two for transitive dynamical systems on a compact metric space \cite{Akin_1996}.

This topological dynamics point of view is adopted in \cite{kurkabook} to study cellular automata. In particular, it is proven that for one-dimensional cellular automata, without the transitivity assumption, there is again a dichotomy between sensitivity and existence of equicontinuity points.
Later, it was shown that it is no longer the case in dimension two or more \cite{Sablik_2010}.

A variant of these notions using averaging along the orbits (mean- or diam-mean- sensitivity/equicontinuity) still presents a dichotomy for transitive dynamical systems \cite{LI_2014}, but no more for one-dimensional cellular automata \cite{DE_LOS_SANTOS_BA_OS_2020}.

Besides, another variant introduced in \cite{Gilman:1987:CLA} taking a measure theoretic point of view also presents a dichotomy in the one-dimensional case.
Interestingly, in a recent work \cite{barbieri2024}, this measure-theoretic dichotomy is studied in the broader context of cellular automata on finitely generated groups (a settings which is classical for cellular automata \cite{celautgrp} but not often considered with the kind of asymptotic dynamical properties discussed here).
The main result of \cite{barbieri2024} builds upon a result from percolation theory to show that the measure-theoretic dichotomy holds if and only if the group is virtually $\Z$.

In the present paper we tackle the question of determining on which groups the topological dichotomy from \cite{kurkabook} does hold.
Our main results are the following:
\begin{itemize}
\item the topological dichotomy holds on any virtually free group (Theorem~\ref{theo:virtuallyfree});
\item if the topological dichotomy holds on some group, then it holds on any subgroup (Corollary~\ref{coro:subgroup});
\item the topological dichotomy does not hold on the free group over two or more generators (Corollary~\ref{coro:freegroup}).
\end{itemize}

\section{Formal definitions and basic facts}

We'll first define cellular automata over Cayley graphs of finitely generated groups following \cite{celautgrp,Meier_2008}, and then introduce the basic notions from topological dynamics to address the main problem tackled in this paper. 

\newcommand\cayley{\mathtt{Cayley}}
Let $G$ be a group generated by some finite set ${E\subseteq G}$ that we assume to be closed under taking inverses.
We denote by $L_g$ the left multiplication by $g$, \textit{i.e.} the map ${h\mapsto gh}$. The Cayley graph $\cayley(G,E)$ is the graph with vertex set $G$ and having for each $e\in E$ the following set of edges labeled by $e$: ${\{(g,g\cdot e): g\in G\}}$. To this graph $\cayley(G,E)$ is associated the distance $d_E$ on $G$ (that is sometimes called the word distance). 

Now fix some finite set $A$ called \emph{alphabet}. We call $A^G$ the set of \emph{configurations}. The group $G$ naturally acts on the left on configurations by ${gx = x \circ L_{g^{-1}}}$ for all $x\in A^G$ and $g\in G$. This action is called the \emph{$G$-shift}. The set of configurations can be endowed with a metric $d^E$, called Cantor metric, defined as follows:
\[\forall x,y \in A^G, d^E(x,y) = 2^{-k  } \textrm{ where } k= min\{  d_E(1_G,g) | g\in G, x(g) \ne y(g)  \}\]
This Cantor metric is always compact and topologically equivalent to the prodiscrete topology (\textit{i.e.} the product over $G$ of the discrete topology on $A$), whatever the choice of generators $E$ (see \cite{celautgrp} for details). More precisely, and as a first step to prove that the properties we study later do not depend on the choice of $E$, let us state the following basic lemma about previously defined metrics.

\begin{lemma}\label{lem:eqdist}
  Let $E_1,E_2$ be two finite generating sets of the same group $G$. Denote by $d^1$ and $d^2 $ the associated Cantor distances over $A^G$. Furthermore, let $B^i(x,r)$ denote the ball of center $x\in A^G$ and radius $r$ for distance $d^i$ (for $i=1,2$). It holds:
    \[\forall \epsilon_1>0, \exists \epsilon_2>0,\forall x \in A^G,  B^2(x,\epsilon_2)\subset B^1(x,\epsilon_1) \]
\end{lemma}

\begin{proof}
  Let us denote by $d_1$ and $d_2$ the word metrics in $G$ associated to $E_1$ and $E_2$ respectively.
  It is known that these distances are equivalent in the following sense (Corollary 11.3 of \cite{Meier_2008}):
  $$ \exists \lambda >1, \forall g,h \in G, \frac{1}{\lambda} \cdot d_1(g,h) \leq d_2(g,h) \leq \lambda \cdot d_1(g,h) $$
  We can suppose without loss of generality that $\lambda$ is an integer.
  For any $\epsilon_1>0$ let ${k = \lceil - log_2(\epsilon_1) \rceil}$ so that we have for $i=1,2$:
  \[\forall x,y \in A^G, (x_{|B_i(1_G,k)}=y_{|B_i(1_G,k)} ) \Rightarrow d^i(x,y) \leq \epsilon_1 .\]
  Now choose $k_2 = k_1 *\lambda $. We then have $ \forall g,h \in G, d_1(g,h)\leq k_1 \Rightarrow d_2(g,h) \leq \lambda \cdot d_1(g,h) \leq \lambda \cdot k_1 = k_2 $, said differently
  $$ \forall g \in G , B_1(g,k_1) \subset B_2(g,k_2) $$
  where notations $B_i$ denote balls in the word metric $d_i$.
  
  Finally, by choosing $\epsilon_2 = 2^{-k_2} $, we can conclude as follows: $\forall x,y \in A^G$, if $d^2(x,y) \leq \epsilon_2$ then $x_{|B_2(1_G,k_2)}=y_{|B_2(1_G,k_2)}$, therefore $x_{|B_1(1_G,k_1)}=y_{|B_1(1_G,k_1)}$ and finally $(d^1(x,y) \leq \epsilon_1$.
\end{proof}

A \emph{cellular automaton} on group $G$ with alphabet $A$ is a map $\Phi : A^G \rightarrow A^G$ that can be defined by a finite set $S \subset G$ called \emph{neighborhood} and a \emph{local map} $ \mu : A^S \rightarrow A$ as follows:
\[\forall g\in G, \forall x \in A^G, \Phi(x)(g) = \mu((g^{-1}x)_{|S}).\]

Cellular automata are exactly the continuous maps for the prodiscrete topology (or any Cantor metric) that commute with $G$-shifts (see \cite{celautgrp}).
They can therefore be seen as topological dynamical systems, and the point of this paper is to study them through classical dynmical properties, namely \emph{sensitivity to initial conditions} and \emph{equicontinuity point}.

First, equicontinuity points are configurations around which the orbits are stable, \emph{i.e.} stay close to each other forever. Precisely, $x\in A^G$ is an equicontinuity point for a cellular automaton $\Phi$ if
\[\forall \epsilon>0, \exists \delta>0, \forall t \in \N, \Phi^t(B^E(x,\delta)) \subset B^E(\Phi^t(x),\epsilon)\]
where $B^E$ denotes the balls in the Cantor metric associated to generating set $E$. 

$\Phi$ is sensitive to initial conditions if there is some fixed precision that can not be guaranteed on any orbit whatever the precision imposed on initial conditions, formally:
 \[\exists\epsilon>0,\forall x \in X, \forall \delta > 0, \exists t \in \N, \exists y \in B(x,\delta) , \Phi^t(y)\notin B(\Phi^t(x),\epsilon).\]

Actually, using classical results of finitely generated groups (Corollary 11.3 of \cite{Meier_2008}), it can be shown that sensitivity to initial conditions and equicontinuity points do not depend on the choice of generating set $E$.

\begin{proposition}
    Let $d^1$ and $d^2$ be the Cantor metrics associated to two generating sets of the same group $G$, and consider some cellular automaton $\Phi$ over $A^G$. Then $x$ is an equicontinuity point for $\Phi$ with metric $d^1$ if and only if it is with metric $d^2$. Similarly, $\Phi$ is sensitive to initial conditions for metric $d^1$ if and only if it is sensitive for metric $d^2$.
\end{proposition}

\begin{proof}
  By symmetry it is enough to show that if $x$ is an equicontinuity with metric $d^2$ then it is also with metric $d^1$.
  Consider any $\epsilon_1>0$. By Lemma~\ref{lem:eqdist} there is $\epsilon_2>0$ such that, for all $y$, $B^2(y,\epsilon_2)\subset B^1(y,\epsilon_1)$.
  $x$ being an equicontinuity point for $d^2$ there exists $\delta_2$ such that $\forall t \in \N, \Phi^t(B^2(x,\delta_2)  ) \subset B^2(\Phi^t(x),\epsilon_2)$.
  Using Lemma~\ref{lem:eqdist} again for $\delta_2$, we obtain $\delta_1>0$ such that, for all $y$, $B^1(y,\delta_1) \subset B^2(y,\delta_2)$.
  Putting everything together we have:
  \[\forall t \in \mathbb{N},
    \Phi^t(B^1(x,\delta_1)) \subset \Phi^t(B^2(x,\delta_2)) \subset B^2(\Phi^t(x),\epsilon_2) \subset B^1(\Phi^t(x),\epsilon_1).\]

  The proof for sensitivity is similar since non-sensitivity is the same formula as above up to the order of quantification.
\end{proof}

It is straightforward from the definition that existence of equicontinuity points implies non-sensitivity.
If the ambient group $G$ is such that the converse holds for any CA, we say that $G$ satisfies \KURKA{}'s dichotomy (in reference to the seminal work of \KURKA{} in the case $G=\Z$, see \cite{kurkabook}).

We now introduce the key concept of \emph{blocking words} in the case of an arbitrary $G$.

\begin{definition}[Blocking word]
  Let $\Phi:A^G\to A^G$ be a CA, $V\subseteq G$ finite and $u:L \rightarrow A$ a pattern of finite support.
  We say $u$ is $V$-blocking for $\Phi$ if
  $$\forall x,y \in A^{G}, 
  (x_{\rvert L} = y_{\rvert L} = u) \Rightarrow 
  (\forall t \in \mathbb{N},  \phi^t(x)_{\rvert V}=\phi^t(y)_{\rvert V} ) $$
\end{definition}

It can be checked that blocking-words can be extended for the support, restricted for the block region, and translated. 
Formally, if a pattern $u$ of support $L$ is $V$ blocking, then, for any ${V'\subseteq V}$, any word ${u'\in A^{L'}}$ with $L\subseteq L'$ and such that $u'_{|L} =  u$ is $V'$-blocking.
Moreover, for any $g\in G$, $g.u$ is a $gV$-blocking word of support $gL$.

\KURKA{}'s dichotomy do not hold in general, however on any $G$ non-sensitivity is equivalent to the existence of blocking words with respect to arbitrarily large balls in the group.

\begin{proposition}\label{prop:nonsensitiveblocking}
  $\Phi : A^G\to A^G$ is not sensitive to initial conditions if and only if, for any $k\in\N$, there exist some $B_G(1_G,k)$-blocking word for $\Phi$.
\end{proposition}

\begin{proof}
  For $k\in\N$ let ${\epsilon_k=2^{-k}>0}$ and denote ${V_k = B_G(1_G,k)}$. $\Phi$ is not sensitive if and only if, for all ${k\in\N}$, the following property $\mathcal{P}_k$ holds:
  \[\exists x_k\in A^G,\exists \delta_k > 0, \forall y \in B(x_k,\delta_k),\forall t, \Phi^t(y) \in B(\Phi^t(x_k),\epsilon_k).\]
  By taking ${p_k = \lceil -log_2(\delta_k) \rceil}$ and ${L_k = B_G(1_G,p_k)}$ and denoting ${u_k}$ the restriction of $x_k$ to domain $L_k$, we can check that the proposition
  \[\forall y \in B(x_k,\delta_k),\forall t, \Phi^t(y) \in B(\Phi^t(x_k),\epsilon_k)\]
  is equivalent to the proposition: $u_k$ is a $V_k$-blocking word of domain $L_k$.
  Property $\mathcal{P}_k$ is therefore equivalent to: there exists some $V_k$-blocking word.
  The proposition follows.
\end{proof}

\section{\KURKA{}'s dichotomy on virtually $\Z$ groups}

Let $G$ be a finitely generated group which is virtually-$\Z$ which means that there is a sub-group $H$ of $G$ of finite index such that there is an isomorphism $\varphi$ from $H$ to $\Z$.
Since $H$ is of finite index in $G$ we know that there is a finite number of right cosets : $ \{ Hg \}_{g \in G} $.
We can choose a fundamental domain, \textit{i.e.} a finite subset $F$ of $G$ such that :
\begin{itemize}
\item $1_G \in F $
\item $G = \bigcup_{f\in F} Hf $
\item For  all two $f$ and $f'$ in $F$, we have $Hf = Hf'$ or $Hf \cap Hf' = \emptyset $
\end{itemize}

Any element $g\in G$ can be uniquely decomposed into a product of an element of $H$ and one of $F$:
${\forall g \in G, \exists! f,z \in F \times H , \space g = zf }$.
The importance of the choice to work with the right cosets will become clear later.

We now introduce a couple of definitions to work with $G$ as if it was $\Z$.
For any $g=zf$ in $G$ where $z\in H $ and $f\in F$, we define ${\p(g) = \varphi(z) \in \mathbb{Z}}$.
This value is well defined thanks to the unicity of the decomposition.
We can think of $\p:G\to \Z $ as a kind of projection, though it is not necessarily a group homomorphism.
In this section we choose ${F\cup\{\varphi^{-1}(1)\}}$ as set of generators for $G$.

\begin{definition}
 For all $k$ in $\Z$ we will call $k$-th \emph{vertebra} the set: 
 $$ V_{\{k\}} = \p^{-1}(\{ k \}) = \{ g \in  G \rvert \p(g) = k \} $$ 
 For $X\subseteq\Z$ we define ${V_X = \p^{-1}(X)}$ and say that $V_X$ is a  \emph{section} if $X$ is a finite interval of $\Z$.
 If $V=V_{[ a,b ]}$ is a section, the length of $V$ is ${l(V) = b - a + 1}$.
 We also define the positive (or right) arm of $V$ as $Br_+(V) = V_{] b, \infty [}$  and the negative (or left) arm of $V$ as $Br_-(V) = V_{] \infty , a [}$.
\end{definition}

\begin{lemma}
  Let $V=V_{[ a,b ]}$ be a section and $h$ an element of $H$, then $hV$ is a section and $ l(hV) = l(V) $. More precisely, $hV= V_{[ \varphi(h) + a , \varphi(h) + b ] }$.
\end{lemma}

\begin{proof}
  It is sufficient to show that  $hV= V_{[ \varphi(h) + a , \varphi(h) + b ] }$.
  We show this by double inclusion.
  
  Let $hv \in hV$, there is $n\in H$ and $f\in F$ such that $hv = hnf$ and $\varphi(n) \in [ a,b ] $ .
  Then $\p(hv) = \p(hnf) = \varphi(hn) = \varphi(h) + \varphi(n) \in [ \varphi(h) + a , \varphi(h) + b ] $.
  Finaly we have then that $hv \in V_{[ \varphi(h) + a , \varphi(h) + b ] }  $ and so  $hV \subset V_{[ \varphi(h) + a , \varphi(h) + b ] } $.

    Now consider $ g\in V_{[ \varphi(h) + a , \varphi(h) + b ] } $.
    With the same method we have that $ h^{-1}g \in V $, so $ g = hh^{-1}g \in hV $.
\end{proof}

The goal of the next definition is to formalize the impact on the ``$\Z$-spine'' of a multiplying by an element in a vertebra.

\newcommand\imp{\text{imp}}

\begin{definition}[Impact]
  Let $s\in G $ and $k\in \mathbb{Z}$, we will call impact the value
  $$ \imp_k (s) = max \{ \rvert p(gs) - p(g)\rvert , g \in V_{\{k\}}\}$$
\end{definition}

\begin{proposition}
     $\imp_k ( s) $ does not depend on $k$.
\end{proposition}

\begin{proof}
  We fix ${s\in G}$.
  Denote ${\psi (g) = \rvert \p(gs) - \p(g)\rvert}$ for any $g\in G$, so that ${\imp_k(s) = \max_{g\in V_{\{k\}}}\psi(g)}$.
  Let $\p \in \mathbb{Z} $, we have $V_{\{\p\}} = L_{\varphi^{-1}(\p)}(V_{\{0\}}) $.
  Let us show that $\forall h \in H, \psi \circ L_h = \psi $ which is enough to conclude the proposition by the previous expression of $\imp_k$.

  Taking $g=zf \in G$  where $z\in H$ and $f\in F$, it holds
  $$\psi \circ L_h ( g) = \psi (hzf) = \rvert \p(hzfs) - \p(hzf)\rvert $$ we know that $fs \in G$ and so $\exists! h',f' \in H \times F, fs = h'f' $.
  Which give us ${hzfs = hzh'f' }$ where $hzh' \in H$ because $H$ is a sub group.

  We thus have:
  \begin{equation} 
    \begin{split}
      \psi \circ L_h(g)  & = \rvert \varphi(hzh') - \varphi(hz) \rvert \\
                         & = \rvert \varphi(h) + \varphi(zh') - \varphi(h) - \varphi(z) \rvert \\
                         & = \rvert \p(zh'f') - \p(zf) \rvert \\
                         & = \rvert \p(gs) - \p(g) \rvert \\
                         & = \psi(g) 
    \end{split}
  \end{equation}
\end{proof}

From this proposition, we denote ${\imp(s)}$ the value of ${\imp_k(s)}$.
We now consider any CA $\Phi: A^G\to A^G$ of neighborhood $S$ and define constant $\Delta = \max \{ \imp(s):s\in S \} $.  The next lemma shows that $\Delta$ is a sufficient length for a section to have its left arm $S$-disconnected from its right arm, and therefore in the presence of a blocking word blocking such a section, the dynamics on the right arm is independent from the dynamics on the left arm.

\begin{lemma}\label{lem:disconnect}
    Let $V= V_{[ a; b ]}$ be a section such that $l(V) \geq  \Delta$ and let $u\in A^L$ be a $V$-blocking word, then $\forall x, y \in A^G$, if ${x_{|L}=y_{|L}=u}$ and $x_{|Br_+(V)}=y_{|Br_+(V)}$ then it holds:
    $$\forall t \in \mathbb{N}, \Phi^t(x)_{|Br_+(V)} = \Phi^t(y)_{|Br_+(V)}. $$
    The symmetric implication holds for $Br_-$.
\end{lemma}
\begin{proof}
  We show the property ${\Phi^t(x)_{|Br_+(V)} = \Phi^t(y)_{|Br_+(V)}}$ by induction on $t$ (the proof is similar for $Br_-$).
  The case ${t=0}$ is just the hypothesis.
  
  Suppose now that the equality holds for $t-1\geq 0$: 
  \[\Phi^{t-1}(x)_{|Br_+(V)} = \Phi^{t-1}(y)_{|Br_+(V)}.\]
  Since $u$ is $V$-blocking, we have ${\Phi^{t'}(x) _{|V}= \Phi^{t'}(y)_{|V} }$ for all $t'\geq 0$.
  Considering ${i\in B_+(V)}$, we have:
  \[\Phi^t(x)(i) = \mu(( \Phi^{t-1}(x) \circ L_i  )_{|S}\]
  where $\mu$ is the local rule of $\Phi$, and the same holds for $y$.

  For any $s\in S$, ${L_i(s)=is}$ and
  \[p(is)\geq p(i)-\imp(s)\geq p(i) - \Delta \geq b+1 - \Delta\geq a.\]
  We thus have
  \[ ( \Phi^{t-1}(x) \circ L_i  )_{|S}=( \Phi^{t-1}(y) \circ L_i  )_{|S}.\]
  We conclude from the above equalities that ${\Phi^t(x)(i) = \Phi^t(y)(i)}$ for all ${i\in B_+}$.
  The lemma follows by induction on $t$.
\end{proof}

On the other hand, the structure of $G$ allows to glue blocking words whose support are sections to form larger blocking words.

\begin{lemma}\label{lem:glueing}
  For $i=1,2$ let $u_i$ be a word of domain $L_i = V_{[ p_i,q_i ] }$ which is $V_i$-blocking for $V_i = V_{[ a_i,b_i ] }$. 
  Suppose moreover $l(V_i) \geq \Delta$ and $ q_1 < p_2 $ and ${a_1\leq b_2}$. 
  Let $x\in A^G$ be any configuration containing both $u_1$ and $u_2$, and define $u= x_{|L}$ for $L= V_{[ p_1,q_2 ]} $.
  Then $u$ is $V$-blocking for $V= V_{[ a_1,b_2 ]} $.
\end{lemma}

\begin{proof}
  Consider any ${y\in A^G}$ with ${y_{|L}=u}$, we define configuration ${z\in A^G}$ by
  \[
    z(g) =
    \begin{cases}
      x(g)&\text{ if }g\in Br_-(L),\\
      u(g)&\text{ if }g\in L,\\
      y(g)&\text{ if }g\in Br_+(L).\\
    \end{cases}
  \]
  We first apply Lemma~\ref{lem:disconnect} with configurations $x$ and $z$ and blocking word $u_2$ which gives
  \[\forall t \in \mathbb{N} , \Phi^t(x)_{|Br_-(V_2)}=\Phi^t(w)_{|Br_-(V_2)}.\]
  Similarly, applying Lemma~\ref{lem:disconnect} for configurations $z$ and $y$ and blocking word $u_1$ we get
  \[\forall t \in \mathbb{N} , \Phi^t(z)_{|Br_+(V_1)}=\Phi^t(y)_{|Br_+(V_1)}.\]
  Combining both equalities above, we deduce
  \[\forall t \in \mathbb{N},\forall i \in V_{] b_1,a_2 [ }, \Phi^t(x)(i) =  \Phi^t(z)(i) = \Phi^t(y)(i).\]
  Finally, since ${V=V_1\cup V_2\cup V_{] b_1,a_2 [ }}$ and since $u_1$ and $u_2$ are subpatterns of $u$ we deduce that the equality ${\Phi^t(x)(i) =  \Phi^t(y)(i)}$ holds for all ${t\in\N}$ and all ${i\in V}$.
  This shows that $u$ is a $V$-blocking word.
\end{proof}

Combining the two previous lemma, we can show that any configuration containing infinitely many blocking words to the left and to the right, and which all block a section of length at least $\Delta$, is actually an equicontinuity point.

\begin{lemma}\label{lem:recglueing}
  Suppose $x\in A^G$ verifies $\forall k \in \N$$ , \exists u_i\in A^{L_i}$ for ${i=1,2}$ with $L_1 \subset Br_-(V_{\{-k\}})$ and $L_2 \subset Br_+(V_{\{k\}})$, and $u_i = x_{|L_i}$ where $u_i$ is $V_i$-blocking with the $V_i$ are sections which verify ${V_1\subseteq Br_{-}(V_{\{-k\}})}$ and ${V_2\subseteq Br_{+}(V_{\{k\}})}$ and $l(V_i) > \Delta$. Then $x$ is an equicontinuity point. 
\end{lemma}

\begin{proof}
  Take any ${\epsilon>0}$ and choose $ k = 1+ \lceil log_2(\epsilon)\rceil $.
  By hypothesis and using Lemma~\ref{lem:glueing}, we have that $x_{|L}$ is a $V$-blocking word where
  \begin{align*}
    L &= L_1 \cup (Br_+(L_1) \cap Br_-(L_2) ) \cup L_2\text{ and}\\
    V &= V_1 \cup (Br_+(V_1) \cap Br_-(V_2) ) \cup V_2.
  \end{align*}
  Choose ${p= \max_{g \in L} \{ |p(g)| \} + 1 }$ so that ${L\subset V_{[ -(p-1) , p-1 ] } \subset B_G(1_G,p)}$ (by our choice of generators for $G$).
  With ${\delta= 2^{-p}}$, any ${y\in B(x,\delta)}$ verifies ${y_{|L}=x_{|L}}$.
  Since $x_{|L}$ is a $V$-blocking word and since ${B_G(1_G,k)\subseteq V}$ (by our choice of generators) we deduce that
  \[\forall t \in \mathbb{N}, \Phi^t(y)\in B(\Phi^t(x),\epsilon).\]
  The lemma follows.
\end{proof}

From the above lemmas and Proposition~\ref{prop:nonsensitiveblocking} we deduce the \KURKA{}'s dichotomy.

\begin{theorem}\label{theo:virtuallyfree}
  If group $G$ is virtually $\Z$, then a CA is sensitive to initial conditions if and only if it has no equicontinuity points.
\end{theorem}
\begin{proof}
  Consider a non-sensitive CA with neighborhood $S$ and let ${\Delta = \imp(S)}$ as above.
  From Proposition~\ref{prop:nonsensitiveblocking} there must exist a $V$-blocking word $u$ with ${l(V)\geq\Delta}$.
  Up to translation and extension we can suppose that the domain $L$ of $u$ is ${L=V_{[0,n]}}$.

  Let us now define a map ${m:G\to L}$ that we will use later to define a configuration covered by translated copies of $u$:
  \[\forall g \in G,  m(g) = \varphi^{-1}(p(g) \bmod n+1) (\varphi^{-1}(p(g)))^{-1} g.\]
  ${m(g)\in L}$ because if ${g=zf}$ with ${z\in H}$ and ${f\in F}$, then $p(g)=\varphi(z)$ and a straightforward computation shows that ${p(m(g))=\varphi(z) \bmod n+1}$.
  Moreover, letting ${h = \varphi^{-1}(n+1)}$ we have ${\forall g\in G, m(hg)=m(g).}$

  We then define configuration $x = u\circ m$ which verifies ${x_{|L} = u}$ and ${x(g)=x(hg)}$ for all ${g\in G}$.
  Concretely, for any ${i\in\Z}$, $x$ contains a translated copy of $u$ at position $h^i$.
  Thus for any $k$, we can find ${i,j\in\Z}$ with ${i+n<-k}$ and ${k<j}$ such that ${x_{|L_1}}$ and ${x_{|L_2}}$ are translated copies of $u$ with ${L_1=V_{[i,i+n]}}$ and ${L_2=V_{[j,j+n]}}$, and which are $V_1$-blocking and $V_2$-blocking respectively with ${V_1\subseteq Br_{-}(V_{\{-k\}})}$ and ${V_2\subseteq Br_{+}(V_{\{k\}})}$.
  In other words, $x$ verifies the hypothesis of Lemma~\ref{lem:recglueing}, so it is an equicontinuity point.
\end{proof}

\section{Lifting a counter-example from a subgroup}

Let $G$ be a group with subgroup $H$. 
Consider a CA on $H$ with alphabet $A$, neighborhood $S\subseteq H$ and local map ${\mu: A^S\to A}$. It defines a global map ${\Phi_H:A^H\to A^H}$.
Since $S\subseteq G$, $\mu$ also defines a global map ${\Phi_G:A^G\to A^G}$.
In this section we show that this ``lift'' from $\Phi_H$ to $\Phi_G$ (called induction in \cite[Section 1.7]{celautgrp}) preserves both non-sensitivity to initial conditions and absence of equicontinuity points.
Therefore whenever $\Phi_H$ is a counter-example to Kurka's dichotomy on $H$ then $\Phi_G$ is a counter-example to Kurka's dichotomy on $G$.

Let us fix a set of generators $E_H$ for $H$ and choose a set of generators $D$ for $G$ that extends $E_H$ (\textit{i.e.} ${E_H\subseteq D}$). We will denote by ${d_H}$ and $d_G$ the distances on $H$ and $G$ associated to this particular choices of generators, and $d^H$ and $d^G$ the corresponding Cantor metrics for configurations.

Note that for our particular choice of generators for any ${h_1,h_2\in H}$ it holds that ${d_G(h_1,h_2)\leq d_H(h_1,h_2)}$.

We now make a particular choice of one representative in each left-coset $gH$ as follows: take the minimum $k$ such that ${gH\cap B_G(1_G,k)\neq\emptyset}$ and pick an element in this intersection (so in particular for $H$ we pick $1_H=1_G$). This way, we get a set ${F\subseteq G}$ of representative with the following properties:
\begin{itemize}
\item $G = \cup_{f \in F} fH $
\item    $\forall f \in F,\forall g \in fH, ||g||_G \geq ||f||_G$ where $||q||_G = d_G(1,q),\forall q \in G$
\item $\forall g \in G, \exists! f,h \in F \times H , g = fh $ 
\end{itemize}

From the above, we then define a projection ${\pi : G\to H}$ by
\[\forall f \in F, \forall g \in fH, \pi(g) = f^{-1}g \in H.\]
This projection allows to define $||g||_H=||\pi(g)||_H$ for any $g\in G$.

We also define $\omega : G\to\N$ by ${\omega(g) = d_G(1,gH)}$
which is equivalent to $\forall f\in F, \forall g \in fH, \omega(g) = ||f||_G$.
A \emph{rectangle} of height $k$ and length $l$ (both integers) is the set ${R(k,l) = \{ g \in G \ |\ \omega(g) \leq k, ||g||_H \leq l  \}}$.

The following proposition shows that rectangles are comparable to balls in $G$.
\begin{proposition}\ 
  \begin{enumerate}
  \item $\forall k,l \in \mathbb{Z} , R(k,l) \subset B_G(1,k+l) $
  \item $\forall  k  \in \mathbb{Z} , B_G(1,k) \subset R(k, \Lambda (k)), \textrm{ where } 
    \Lambda(k) = \max_{g \in B_G(1,k)} ||g||_H $
  \end{enumerate}
\end{proposition}

\begin{proof}
  For the first item, let $g\in R(k,l)$, $\exists! f,h \in F \times H , g = fh$. By definition of $\omega(g)$ there are $e_1 , e_2 , ..., e_{\omega(g) } \in D$ with $e_1e_2...e_{\omega(g)} = f  $.  Besides, there are $b_1,b_2,...,b_{||h||_H} \in E_H$ with $b_1b_2...b_{||h||_H} = h$, therefore we can write
  \[g = fh = e_1e_2...e_{\omega(g)}b_1b_2...b_{||h||_H}.\]
Since $\omega(g) \leq k$ and $||h||_H \leq l$, we deduce  ${ ||g||_G \leq k+l }$.
  
  For the second item let us show $\forall g \in G , ||g||_G \geq \omega(g)$. Let $g=fh$ be the coset decomposition with $f\in F$ so that $\omega(g) = \omega(f) = ||f||_G$. Moreover, our choice of $F$ ensures that $ \forall g \in fH, ||g||_G \geq ||f||_G$. Therefore $\omega (g) > k $ implies $g \notin B_G(1,k)$.
  We deduce that $\forall g \in B_G(1,k), \omega(g) \leq  k$  and ${ ||g||_H \leq \Lambda(k) }$ by definition of $\Lambda (k)$.
\end{proof}

We can now relate the dynamics of $\Phi_G$ to that of $\Phi_H$: $\Phi_G$ is just the parallel execution of $\Phi_H$ on each coset.

\begin{lemma}[Lifted dynamics]\ 
\label{lem:parallel}
  \begin{enumerate}
  \item $\forall f_1,f_2 \in F , \forall x_1,x_2 \in A^G$ : 
    $$ (f_1^{-1}x_1)_{|H} = (f_2^{-1}x_2)_{|H} \Rightarrow \forall t \in \mathbb{N}, (f_1^{-1}\Phi_G^t(x_1))_{|H} = (f_2^{-1}\Phi_G^t(x_2))_{|H} $$
  \item $\forall x\in A^G, \forall t\in\N, \Phi_G^t(x)_{|H} = \Phi_H^t(x_{|H})$
  \item $ \forall x\in A^G, \forall f\in F,\forall t\in\N,
    (f^{-1}\Phi^t_G(x))_{|H} = \Phi^t_H((f^{-1}x)_{|H})$
  \end{enumerate}
\end{lemma}

\begin{proof}
  We show first item by induction on $t$: the conclusion of the implication holds for $t=0$, and if $(f_1^{-1}\Phi_G^t(x_1))_{|H} = (f_2^{-1}\Phi_G^t(x_2))_{|H}$ for some $t$, then for any $h\in H$, $\forall i \in \{ 1,2 \}$: 
    \[f_i^{-1}\Phi_G^{t+1}(x_i)(h)= \Phi ( f_i^{-1}\Phi_G^t(x_i) )(h) = \mu( (h^{-1}f_i^{-1}\Phi_G^t(x_i))_{|S}  ) \]
    For $s\in S \subset H$, $$ h^{-1}f_1^{-1}\Phi_G^t(x_1) (s) = f_1^{-1}\Phi_G^t(x_1)(hs) = f_2^{-1}\Phi_G^t(x_2)(hs) = h^{-1}f_2^{-1}\Phi_G^t(x_2) (s) $$ 
    because $hs \in H$. By definition of $\Phi_G$ from $\mu$ we deduce $(f_1^{-1}\Phi_G^{t+1}(x_1))_{|H} = (f_2^{-1}\Phi_G^{t+1}(x_2))_{|H}$.

  The second item is shown by induction again. Suppose that for some $t\in\N$ $x\in A^G$ it holds $\Phi_G^t(x)_{|H} = \Phi_H^t(x_{|H})$ and let $h\in H$. Since $S\subset H$, $\forall y \in A^G, (y\circ L_h)_{|S} = (y_{|H} \circ L_h )_{|S}  $ because $\forall s \in S ,hs \in H$.  We therefore have: 
  \begin{equation}
    \begin{split}
      \Phi^{t+1}_G(x)(h) 
      &= \mu((h^{-1}\Phi^t_G(x))_{|S})  
      \\&= \mu((\Phi^t_G(x) \circ L_h  )_{|S}) 
      \\&= \mu((\Phi^t_G(x)_{|H} \circ L_h  )_{|S})  
      \\&= \mu((\Phi^t_H(x_{|H}) \circ L_h  )_{|S})  
      \\&= \mu((h^{-1}\Phi^t_H(x_{|H})   )_{|S}) 
      \\&= \Phi^{t+1}_H(x_{|H})(h) 
    \end{split}    
  \end{equation}

  The third item follows directly from the two first ones.
\end{proof}

One can prove that the projection onto $H$ of equicontinuity points of $\Phi_G$ gives equicontinuity points for $\Phi_H$.

\begin{proposition}
  If $\Phi_G$ possesses an equicontinuity point then so does $\Phi_H$.
\end{proposition}

\begin{proof}
Let $x$ be an equicontinuity point for $\Phi_G$ and define $x' = x_{|H}$. Consider $\epsilon > 0$. Note that $\forall v,w \in A^G, (d^G(v,w) \leq \epsilon) \Rightarrow (d^H(v_{|H},w_{|H}) \leq \epsilon) $  because for any $k\in\N$,  $B_H(1,k)\subset B_G(1,k)$. By hypothesis on $x$ there is $\delta_G >0$ such that $\forall y \in A^G, d^G(x,y) \leq \delta_G \Rightarrow (\forall t,\Phi_G^t(y) \in B^G(\Phi_G^t(x),\epsilon) ) $.
Let $p_G = \lceil -log_2(\delta_G) \rceil$, and let $p_H$ be such that $B_G(1,p_G) \cap H \subset B_H(1,p_H)$. We choose $\delta_H = 2^{-p_H}$.
For any $y\in B^H(x',\delta_H)$, consider $y' \in A^G$ such that $y'_{|H}= y$ et $y'_{|G\setminus H }=x_{|G\setminus H }$. We have $d^G(y',x) \leq \delta_G$ because on one hand $B_G(1,p_G) \cap H \subset B_H(1,p_H)$, and on the other hand $Y$ and $x$ coincide outside $H$. We therefore have:
\[\forall t \in \N , d^G( \Phi_G^t(y'),\Phi_G^t(x) ) \leq \epsilon \textrm{ which gives } d^H( \Phi_G^t(y')_{|H},\Phi_G^t(x)_{|H} ) \leq \epsilon\]
By Lemma~\ref{lem:parallel} we conclude
\[\forall t \in \mathbb{N} , d^H(\Phi_H^t(y),\Phi_H^t(x')) \leq \epsilon.\]
\end{proof}

With appropriate choices of rectangles when translating distances into regions, one can also prove that the sensitivity of $\Phi_G$ transfers to $\Phi_H$.

\begin{proposition}
  If $\Phi_G$ is sensitive to initial conditions then so is $\Phi_H$.
\end{proposition}

\begin{proof}
  By contradiction, we suppose $\Phi_G$ is sensitive with sensitivity constant $\epsilon_G>0$  but that $\Phi_H$ is not sensitive.
  Let $k_G = \lceil - log_2(\epsilon_G) \rceil $ and $k_H = \Lambda(k_G)$.
  Observe that $B_G(1,k_G) \subset R(k_G,k_H)= R $ and let $\epsilon_H = 2^{-k_H}$. 
  Since $\Phi_H$ is not sensitive:
\[\exists x\in A^H, \exists \delta_H > 0, \forall y \in B^H(x,\delta_H),\forall t, d^H(\Phi_H^t(y) , \Phi_H^t(x)) \leq \epsilon_H\]
We now define $x'= x \circ \pi$ and seek a contradiction with $\epsilon_G$-sensitivity of $\Phi_G$ at $x'$.
 First, let $p_H = \lceil - log_2(\delta_H) \rceil$.
 Note that $R(k_G,p_H) \subset B_G(1,k_G+p_H) $ and define $p_G = k_G + p_H$ and $\delta_G = 2^{-p_G} $.
 Pick any $y' \in B^G(x',\delta_G)$.
 By definition, $y'_{|B_G(1,p_G)} = x'_{|B_G(1,p_G)} $ therefore $y'_{|R(k_G,p_H)} = x'_{|R(k_G,p_H)} $.
 We deduce:
 \[\forall f \in F \cap B_G(1,k_G), (f^{-1}y')_{|H} \in B^H(x,\delta_H)\]
By the non-sensitivity assumption on $H$ we get:
$$\forall f \in F \cap B_G(1,k_G), \forall t \in \N,
(f^{-1}\Phi_G^t(y'))_{|H} = \Phi_H^t((f^{-1}y')_{|H}) \in B^H(\Phi_H^t(x),\epsilon_H).$$
Therefore $\forall f \in F \cap B_G(1,k_G), \forall t \in \mathbb{N}$, it holds
\begin{equation}
    \begin{split}
(f^{-1}\Phi_G^t(y'))_{|B_H(1,k_H)} 
  & = \Phi_H^t(x)_{|B_H(1,k_H)} \\
  &= \bigl(\Phi_H^t((f^{-1}x')_{|H})\bigr)_{|B_H(1,k_H)} \\
  &= \bigl((\Phi_G^t(f^{-1}x'))_{|H}\bigr)_{|B_H(1,k_H)} \textrm{ (by Lemma~\ref{lem:parallel})} \\
  &= \bigl(f^{-1}\Phi_G^t(x')\bigr)_{|B_H(1,k_H)}.
    \end{split}
\end{equation}

The last equality comes from the fact that $\forall f \in F, (f^{-1}x')_{|H} = x'_{|H}=x$, because $x'(fh) = x(\pi(fh)) = x(h)$ for $h\in H$ and $ f \in F$.
The chain of equality above can be rewritten:
$$\Phi_G^t(y')_{|fB_H(1,k_H)} =  \Phi_G^t(x')_{|fB_H(1,k_H)}$$
By definition of rectangles, we have
$\forall h,l \in \mathbb{N}^2, R(h,l) = \cup_{f \in F \cap B_G(1,h)} (fB_H(1,l))   $
so our previous equality true for any ${f\in F\cap B_G(1,k_G)}$ actually gives
$$\forall t\in \mathbb{N}, \Phi^t(y')_{|R(k_G,k_H)} =  \Phi^t(x')_{|R(k_G,k_H)}.$$
Since $B_G(1,k_G) \subset R(k_G,k_H) $,  we finally have
$$\forall t\in \mathbb{N}, \Phi^t(y') \in  B^G(\Phi_G^t(x'),\epsilon_G). $$
Since $y'$ was arbitrarily picked inside $B^G(x',\delta_G)$, this contradicts $\epsilon_G$-sensitivity of $\Phi_G$.
\end{proof}

\begin{corollary}\label{coro:subgroup}
  If Kurka's dichotomy holds on a group, then it must also hold on any of its subgroup.
\end{corollary}

\section{Free group with 2 or more generators}

\newcommand\interior{\iota}
\newcommand\boundary{\beta}
\newcommand\NZ{\text{D}}
In this section we show that Kurka's dichotomy does not hold on the free group with $2$ (or more) generators.
Let $G$ be the free group with generators ${E=\{a,a^{-1},b,b^{-1}\}}$ and $A=\{0,1,\interior,\boundary\}$. Intuitively, the CA we are going to define on state set $A$ behaves like the addition modulo $2$ on $\{0,1\}$ and, besides, it allows the existence of pervasive ``obstacles'' which are zones whose \emph{interior} is filled with $\interior$ and whose \emph{boundary} is in state $\boundary$.
The rest of the behavior consists in erasing states ${\{\interior,\boundary\}}$ which are not part of a well-formed obstacle.
To precisely define this behavior, we need the following definitions.

In a configuration ${c\in A^G}$, a position ${g\in G}$ is said \emph{free} if the two following conditions hold:
\begin{itemize}
\item ${c_g\in\{0,1\}}$ and
\item ${|\{g'\in g\cdot E : c_{g'}\in\{0,1\}\}|\geq 2}$
\end{itemize}
A position $g$ is said \emph{blocked} in $c$ if one of the following conditions hold:
\begin{itemize}
\item ${c_g=\interior}$ and ${\forall g'\in g\cdot E}$, ${c_{g'}\in\{\interior,\boundary\}}$, or
\item ${c_g=\boundary}$ and exactly one element ${g'\in g\cdot E}$ is such that ${c_{g'}= \interior}$ and all other elements of ${g\cdot E}$ are free in $c$.
\end{itemize}

Now define the CA $F$ with alphabet $A$ as follows:
\[ F(c)_g =
\begin{cases}
  \displaystyle c_g+\sum_{g'\in g\cdot E:c_{g'}\in\{0,1\}}c_{g'} \bmod 2 &\text{ if }c_g\in\{0,1\},\\
  c_g&\text{ if $g$ is blocked in $c$},\\
    0&\text{ else.}
\end{cases}
\]

From the definition of 'free' and 'blocked' above, it can be checked that $F$ has radius $2$. For any configuration $c$ define ${\NZ(c) = \{g\in G : c_g\not\in\{0,1\}\}}$.

Some facts that follow directly from the above definitions:
\begin{itemize}
\item $\NZ(c)$ is decreasing for inclusion under the action of $F$, and if all positions ${g\in \NZ(c)}$ are blocked in $c$, then ${\NZ(c)=\NZ(F(c))}$.
\item A free position stays free forever.
\item If ${c_g\in\{0,1\}}$ and $g$ is not free in $c$, then all positions ${g'\in g\cdot E}$ such that ${c_{g'}\in\{\interior,\boundary\}}$ change their state in one step and $g$ becomes free in $F(c)$.
\item In a configuration where all positions ${g\in \NZ(c)}$ are blocked, all ${g\not\in \NZ(c)}$ are free and $F$ is acting like an Abelian CA on them: formally, if
  \[X_c = \{c' : \NZ(c)=\NZ(c') \text{ and } \forall g\in \NZ(c), c'_g=c_g\}\]
  then 
  \[\forall c^1,c^2\in X_c,\forall g\not\in \NZ(c) : F(x^1\oplus x^2)_g=F(x^1)_g+F(x^2)_g\bmod 2\]
\end{itemize}
where ${(x^1\oplus x^2)(g) = x^1_g + x^2_g \bmod 2.}$

We are now ready to analyze the dynamical properties of $F$.

\begin{lemma}
  $F$ is not sensitive to initial conditions.
\end{lemma}
\begin{proof}
  For any $n$, one can define a configuration $c$ where all positions ${\|g\|< n}$ are blocked and all other positions are free: 
  \[c_g =
  \begin{cases}
    \interior&\text{ if }\|g\|\leq n-2,\\
    \boundary&\text{ if }g=n-1,\\
    0&\text{ else.}
  \end{cases}
\] If $c'$ is $2^{-n}$-close to $c$ then ${F^t(c')}$ will stay $2^{-n}$-close to ${F^t(c)}$ for any $t$, showing that ${2^{-n}}$ cannot possibly be a sensitivity constant for $F$. This works for all $n$ so $F$ is not sensitive.
\end{proof}

\begin{lemma}
  $F$ has no equicontinuity point.
\end{lemma}
\begin{proof}
  Suppose by contradiction that $c$ is an equicontinuity point. We consider $2$ cases.
  
  \textbf{First case:} for any $g\in G$ and any time $t$, position $g$ is blocked in ${F^t(c)}$: this implies that $c$ is uniformly in state $\interior$. Indeed, no position in state ${\{0,1\}}$ can be blocked, and for a position in state $\boundary$ to be block, there must be states in ${\{0,1\}}$ in the neighborhood. Now for any ${n\geq 0}$ one can consider $c^n$ which is equal to $c$ except at some position $g$ with ${\|g\|=n}$ where ${c^n_g=0}$. It is straightforward to check that ${F^n(c)_{1_G}=\interior}$ while ${F^n(c^n)_{1_G}\in\{0,1\}}$. This being for any $n$, we have a contradiction with the hypothesis that $c$ is an equicontinuous point.

  \textbf{Second case:} there is some position which is not blocked at some time in the orbit. This is the same as supposing ${F^t(c)_g\in\{0,1\}}$ for some $g$ and some $t$ (because in one step a non-blocked position turns to some state from ${\{0,1\}}$).
  Now consider any large enough $n\geq 0$, precisely ${n>\|g\|+2t}$, and let $d$ be the configuration with ${d_{g'}=c_{g'}}$ for ${\|g'\|\leq n}$ and ${d_{g'}=0}$ for ${\|g'\|> n}$. Note that by choice of $n$ and since the radius of $F$ is $2$, we also have ${F^t(d)_g\in\{0,1\}}$. 
  Moreover, the set $\NZ(d)$ is finite by construction and decreasing by application of $F$ so there is a time ${T}$ after which it stabilizes: ${\NZ(F^{T+1}(d)) = \NZ(F^T(d))}$.
  As said above, for any ${t'\geq T}$, all positions ${g'\not\in \NZ(F^T(d))}$ are free in ${F^{t'}(d)}$.
  Moreover they all have at least $2$ neighbors which are not in ${\NZ(F^T(d))}$ and therefore which are free.
  Thus we can construct a $E$-connected path ${(g_i)}$ of positions with ${\|g_i\|=\|g\|+i}$ such that ${g_0=g}$ and ${g_i}$ is free in ${F^{t'}(d)}$ as soon as ${t'\geq T}$.
  Taking $i$ large enough so that ${\|g_i\|> n}$, let's consider configuration $d'$ which is equal to $d$ except at some position ${g_d = g_i \cdot g'}$ with ${\|g_d\|=\|g_i\| + T}$ where ${d'_{g_d}=1}$ (whereas ${d_{g_d}=0}$).
  During the $T$ first steps, $d$ and $d'$ evolve in the same way in the central positions up to radius $n$ and differ only in a region filled with states in ${\{0,1\}}$ where the Abelian behavior of $F$ applies as said above.
  Therefore, at time $T$, we have the following:
  \begin{itemize}
  \item ${F^T(d)_{g'}=F^T(d')_{g'}}$ for ${\|g'\|<\|g_i\|}$, and
  \item ${F^T(d)_{g_i}\neq F^T(d')_{g_i}}$ (because $d$ and $d'$ differ only at $g_d$ and this difference reaches $g_i$ exactly after $T$ steps by construction).
  \end{itemize}
  From step $T$ on, the set of positions in state ${\{0,1\}}$ is stabilized, precisely: ${D(F^T(d))=D(F^T(d'))=D(F^{T+1}(d))=D(F^{T+1}(d'))}$. Therefore the Abelian behavior of $F$ applies on free positions and the difference between $d$ and $d'$ at $g_i$ propagate along path ${(g_j)_{0\leq j\leq i}}$. More precisely the closest difference to the center at time ${T+k}$ is at position ${g_{i-k}}$ until it reaches position $g$ at time ${T+i}$.
  To sum up, there is a fixed position $g\in G$ such that for any large enough $n$ we can construct two configurations which are ${2^{-n}}$-close to $c$ and which differ at position $g$ at some time step. This is a contradiction with the hypothesis that $c$ is an equicontinuous point.
\end{proof}

\begin{corollary}\label{coro:freegroup}
  For any $d\geq 2$, Kurka's dichotomy fails on the free group with $d$ generators.
\end{corollary}

\bibliographystyle{acm}
\bibliography{sample.bib}

\begin{thebibliography}{10}

\bibitem{Akin_1996}
{\sc Akin, E., Auslander, J., and Berg, K.}
\newblock When is a transitive map chaotic?
\newblock In {\em Convergence in Ergodic Theory and Probability}. {DE}
  {GRUYTER}, dec 1996, pp.~25--40.

\bibitem{DE_LOS_SANTOS_BA_OS_2020}
{\sc Ba{\~{n}}os, L. D. L.~S., and Garc{\'{i}}a-Ramos, F.}
\newblock Mean equicontinuity and mean sensitivity on cellular automata.
\newblock {\em Ergodic Theory and Dynamical Systems 41}, 12 (nov 2020),
  3704--3721.

\bibitem{barbieri2024}
{\sc Barbieri, S., García-Ramos, F., and Taati, S.}
\newblock Cellular automata, percolation and dynamical dichotomies, 2024.

\bibitem{celautgrp}
{\sc Ceccherini-Silberstein, T., and Coornaert, M.}
\newblock {\em Cellular Automata and Groups}.
\newblock Springer, 2010.

\bibitem{Gilman:1987:CLA}
{\sc Gilman, R.~H.}
\newblock Classes of linear automata.
\newblock {\em Ergodic Theory and Dynamical Systems 7}, 1 (1987), 105--118.

\bibitem{hedlund}
{\sc Hedlund, G.~A.}
\newblock Endomorphisms and {A}utomorphisms of the {S}hift {D}ynamical
  {S}ystems.
\newblock {\em Mathematical Systems Theory 3}, 4 (1969), 320--375.

\bibitem{kurkabook}
{\sc K\r{u}rka, P.}
\newblock {\em Topological and symbolic dynamics}.
\newblock Soci\'et\'e Math\'ematique de France, 2003.

\bibitem{LI_2014}
{\sc Li, J., Tu, S., and Ye, X.}
\newblock Mean equicontinuity and mean sensitivity.
\newblock {\em Ergodic Theory and Dynamical Systems 35}, 8 (aug 2014),
  2587--2612.

\bibitem{Meier_2008}
{\sc Meier, J.}
\newblock {\em Groups, Graphs and Trees}.
\newblock Cambridge University Press, jul 2008.

\bibitem{Sablik_2010}
{\sc Sablik, M., and Theyssier, G.}
\newblock Topological dynamics of cellular automata: Dimension matters.
\newblock {\em Theory of Computing Systems 48}, 3 (apr 2010), 693--714.

\end{thebibliography}
\end{document}